\documentclass[11pt,a4paper]{amsart}
\usepackage{color,latexsym,amsfonts,amssymb,bbm,comment,amsmath, amsthm, bbold}

\usepackage{float}
\usepackage{indentfirst} 
\usepackage{amsmath, amsthm, amssymb,xfrac, mathrsfs} 
\usepackage{tikz} 

\usepackage{hyperref} 
\hypersetup{
colorlinks=true,
citecolor=black!50!red,
linkcolor=black!50!red,
linktoc=all
}
\usepackage[all]{hypcap} 

\newcounter{constant}
\newcommand{\nc}[1]{\refstepcounter{constant}\label{#1}}
\newcommand{\uc}[1]{c_{\textnormal{\tiny \ref{#1}}}}
\newtheorem{theo}{Theorem}[section]

\newtheorem{lemma}[theo]{Lemma}

\newtheorem{cor}[theo]{Corollary}

\numberwithin{equation}{section} 

\theoremstyle{definition}

\newtheorem{question}[theo]{Question}

\newcommand{\R}{\mathbb{R}}

\newcommand{\N}{\mathbb{N}}
\newcommand{\Z}{\mathbb{Z}}
\newcommand{\charf}[1]{\mathbf{1}_{#1}}

\DeclareMathOperator{\dist}{dist}

\begin{document}

\title[Fire retainment on Cayley graphs]{Fire retainment on Cayley graphs}

\author{Gideon Amir}
\address{Bar-Ilan University, Ramat Gan 52900, Israel.}
\email{gidi.amir@gmail.com}

\author{Rangel Baldasso}
\address{Leiden University, 2300 RA Leiden, The Netherlands.}
\email{r.baldasso@math.leidenuniv.nl}

\author{Maria Gerasimova}
\address{Westfälische Wilhelms-Universität Münster, 48149 Münster, Germany.}
\email{mari9gerasimova@mail.ru}

\author{Gady Kozma}
\address{The Weizmann Institute of Science, Rehovot 76100, Israel.}
\email{gady.kozma@weizmann.ac.il}

\begin{abstract}
  We study the fire-retaining problem on groups, a quasi-isometry invariant\footnotemark\ introduced by Mart\'{i}nez-Pedroza and Prytu{\l}a~\cite{mpp}, related to the firefighter problem.
  We prove that any Cayley graph with degree-$d$ polynomial growth does not satisfy $\{f(n)\}$-retainment, for any $f(n) = o(n^{d-2})$, matching the upper bound given for the firefighter problem for these graphs.
  In the exponential growth regime we prove general lower bounds for direct products  and wreath products.
  These bounds are tight, and show that for exponential-growth groups  a wide variety of behaviors is possible. In particular, we construct, for any $d\geq 1$, groups that satisfy $\{n^{d}\}$-retainment but not $o(n^d)$-retainment, as well as groups that do not satisfy sub-exponential retainment.
  
  \medskip
  

\end{abstract}
	
	\maketitle


\section{Introduction}
\stepcounter{footnote}
\footnotetext{
  See Section~\ref{s:notation} for some caveats.}

\par Let $G$ be an infinite connected graph and $\{f(n)\}$ be a sequence of positive integers.
  Assume that a fire breaks out in a finite subset $F_{0}$ of the vertices of $G$.
  At each unit of time $n$, at most $f(n)$ vertices that are not on fire are declared protected and the fire spreads to all non-protected vertices that neighbor at least one vertex already on fire.
  Denote by $U$ the collection of vertices that are never on fire.
  The graph $G$ is said to have the $\{f(n)\}$-retaining property if, for every choice of initially burnt vertices $F_{0}$, there exists a strategy for protecting at most $\{f(n)\}$ vertices at step $n$ such that the set $U$ has growth function equivalent to the growth function of $G$.
  
We use the standard equivalence relation between functions in which $f\preceq g$ if there exists a constant $C>0$ s.t. $f(n)\leq C g(Cn)+C$ for all $n$ and $f\simeq g$ if $f\preceq g$ and $g\preceq f$. It is well-known that up to this equivalence relation, the growth rate of a group is a quasi-isometry invariant.

  The notion of fire retaining is in fact a quasi-isometry geometric invariant (for monotone non-decreasing functions $f$) that was introduced by Mart\'{i}nez-Pedroza and Prytu{\l}a~\cite{mpp} as a variant of the firefighter problem first proposed by Hartnell~\cite{hartnell}.
  In the firefighter problem, one looks for strategies such that the set $G \setminus U$ of vertices that are eventually on fire is finite.
  Whenever this is possible for every finite initial set $F_{0}$, one says that $G$ has the $\{f(n)\}$-containment property.
  Of course, if $G$ has the $\{f(n)\}$-containment property, it also has the $\{f(n)\}$-retainment property.

  The graph $G$ is said to satisfy polynomial containment of degree $d$ (respectively, polynomial retainment of degree $d$) if it has the $\{f(n)\}$-containment property (respectively, $\{f(n)\}$-retainment property) for $f(n)=Kn^{d}$, for some $K \geq 0$.
  Dyer, Mart{\'\i}nez-Pedroza, and Thorne~\cite{dmt} established that having polynomial containment of degree $d$ is a quasi-isometry invariant, while the analogous statement for polynomial retainment of degree $d$ was proved in~\cite{mpp}.

  The following result was established in~\cite{dmt}.
\begin{theo}\textnormal{(\cite[Theorem 2.3]{dmt})}
  Let $G$ be a connected graph with polynomial growth of degree at most $d$.
  Then $G$ satisfies polynomial containment of degree $d-2$.
\end{theo}

  Develin and Hartke~\cite{dh} conjectured the converse should hold for the $d$-dimensional integer lattice, and this was recently verified to hold for Cayley graphs of groups with polynomial growth by Amir, Baldasso, and Kozma~\cite{abk}.

\begin{theo}\label{t:no_polynomial_containment}\textnormal{(\cite[Theorem 1]{abk})}
  Let $G$ be the Cayley graph of a group with polynomial growth of degree at least $d \geq 2$.
  Then $G$ does not have the $\{f(n)\}$-containment property, for any $f(n)=o(n^{d-2})$.
\end{theo}

  Our first result in the present paper is an analogue to Theorem~\ref{t:no_polynomial_containment} for the retaining property, which in particular answers Question 1.5 of~\cite{mpp}.
\begin{theo}\label{t:fire_retaining}
  Let $G$ be the Cayley graph of a group with polynomial growth of order $d$.
  Then $G$ does not have the $\{f(n)\}$-retaining property for any $f(n)=o(n^{d-2})$.
\end{theo}

  The result above in particular implies that the notions of fire retainment and fire containment behave similarly on Cayley graphs of groups with polynomial growth.

\bigskip
  The fire-retaining problem on Cayley graphs with faster-than-polynomial growth presents a richer behavior due to the lack of monotonicity, since the fact that $G$ has the $\{f(n)\}$-retainment property does not necessarily imply an analogous statement for subgraphs of $G$.
  In particular, the two notions of fire retainment and fire containment can behave very differently.
  For free groups, for example, the fire-retaining problem is solved by protecting only one vertex, while fire containment can only be achieved by functions $f$ with exponential growth (\cite{ln}).
  More than that, it is not hard to see that the product $F_{k} \times \Z^{d}$ between the free group with $k$ generators and $\Z^{d}$ has polynomial retainment of degree $d-1$.
  The second contribution of the current paper is a general lower bound for retainment in direct products, which in particular implies that for  $F_{k} \times \Z^{d}$ this is indeed the best possible.

\begin{theo}\label{t:product_groups}
  Let $G$ and $H$ be groups with exponential and subexponential-growth, respectively.
  The group $G \times H$ does not have the $\{f(n)\}$-retainment property for any function $f$ satisfying
\begin{equation}\label{eq:slow_growth_f}
  \sum_{k=1}^{10n}f(k) = o(v_{H}(n)),
\end{equation}
where $v_{H}(n) = |B^{H}_{n}|$ denotes the growth function of $H$.
\end{theo}

  As a corollary, we obtain a negative result for retainment in the case when $H$ has polynomial growth.
\begin{cor}
  Assume that $G$ has exponential growth and $H$ has polynomial growth of degree $d$.
  Then $G \times H$ does not have the $\{f(n)\}$-retainment property for any function $f$ satisfying $f(n) = o(n^{d-1})$.
\end{cor}

  In Theorem~\ref{t:product_groups}, we do not have any restrictions for the retainment of the group $G$.
  It might be, as in the case of the free group $F_k$, that only a constant number of vertices needs to be protected in order to guarantee fire retaining. One may hope that if more vertices are needed to attain fire-retaining on $G$, then this can be translated into a better lower bound for $G \times H$. This prompts us to propose the following question:
\begin{question}\label{q:product_groups_2}
  Assume that $G$ and $H$ are groups with exponential and subexponential growth, respectively.
  Assume that $G$ does not have the $\{g(n)\}$-retainment property and let $f$ be a function satisfying~\eqref{eq:slow_growth_f}.
  Under these hypotheses, is it true that $G \times H$ does not have the $\{f(n) \cdot g(n)\}$-retaining property?
\end{question}
An interesting special case of Question \ref{q:product_groups_2} is that of taking a direct product of $\Z$ with the hyperbolic disc. The hyperbolic disk itself has a simple strategy that allows to retain a constant proportion of the disc by protecting only 2 vertices in each turn, by simply following arbitrary different geodesics to infinity (see Figure \ref{fig:hyp}). It is well-known there are many groups, such as fundamental groups of a surface groups of genus $g\geq 2$ that are quasi-isometric to the hyperbolic disc (see e.g.~\cite{Kapovich_boundariesof}). Thus by quasi-isometry invariance of fire-retaining there are (hyperbolic) groups of exponential growth for which protecting a constant number of vertices each round is enough for retaining, but protecting only a finite amount of vertices overall is not.
\begin{figure}[b]
\includegraphics[scale=0.5]{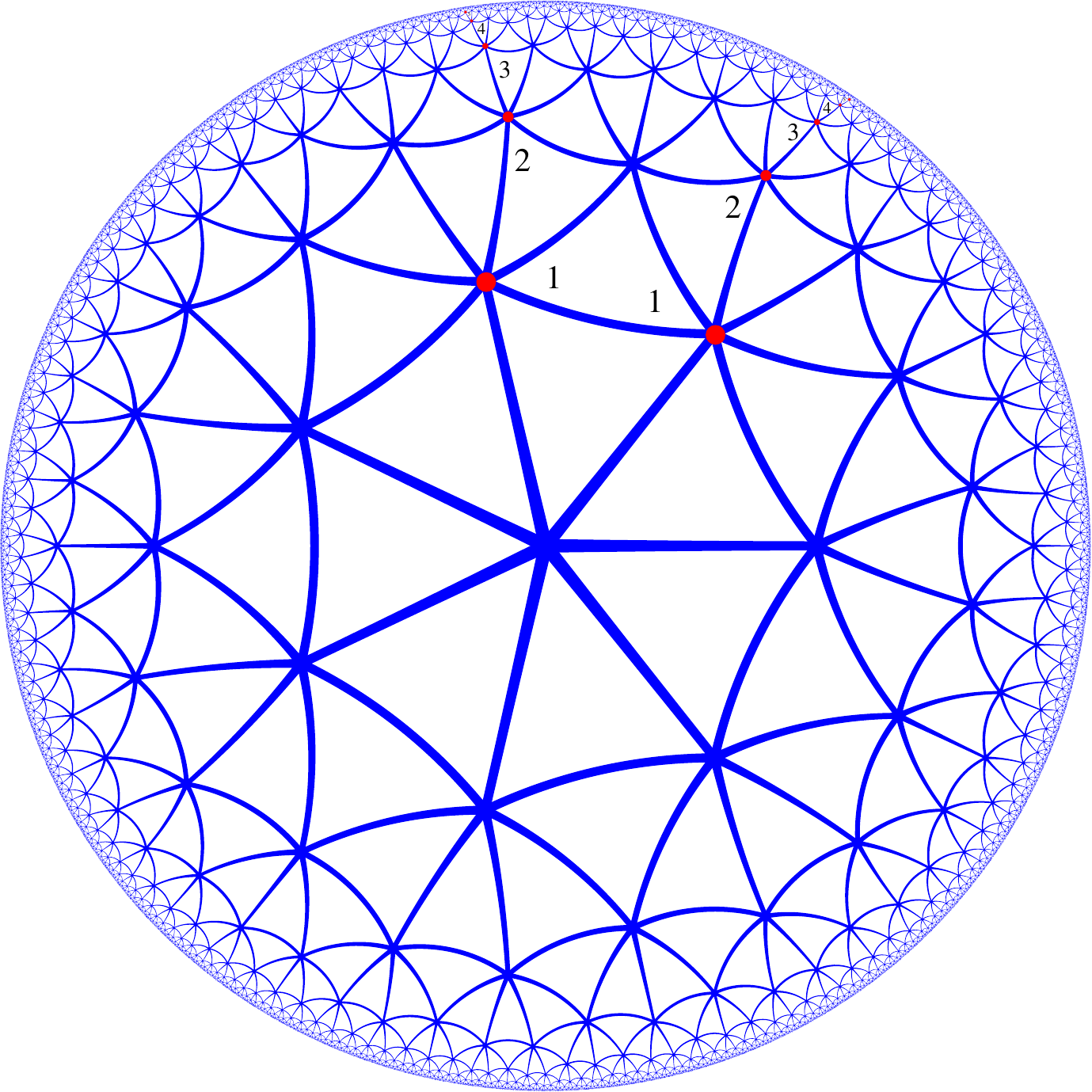}
\caption{\label{fig:hyp}The hyperbolic disk - a seven regular triangulation. The vertices in red are (the beginning of) the set of vertices we protect, and the numbers indicate on which turn we protect them. 
}
\end{figure}

  We say that a function $f$ has subexponential growth if
\begin{equation}\label{eq:slow_growth_f_2}
\lim_{n} \frac{1}{n} \log f(n) =0.
\end{equation}
When both $G$ and $H$ have exponential growth, we prove that the product does not have the $\{f(n)\}$-retaining property, for any $f$ with subexponential growth.
\begin{theo}\label{t:product_groups_3}
  If $G$ and $H$ are groups with exponential growth, then $G \times H$ does not have the $\{f(n)\}$-retainment property for any function $f$ with subexponential growth.
\end{theo}
Note that since any two exponential growth functions are equivalent, by the definition of retainment any finitely generated group $G$ satisfies $\{f(n)\}$-retainment for any function $f$ of exponential growth, thus the above theorem is tight.
\bigskip


  Our last theorem considers wreath products (sometimes known as ``lamplighter'' groups, see Section~ \ref{s:lamplighter} for background and definitions).
  We prove that any non-trivial wreath product  does not have the retainment property for any subexponential function.
\begin{theo}\label{t:lamplighter}
For any infinite group $G$ and any group $H$ with more than one element, the wreath product $H \wr G$ does not satisfy $\{f(n)\}$-retainment for any subexponential function $f$.
\end{theo}
This theorem is tight in the same sense as Theorem~\ref{t:product_groups_3} and its proof can be found in Section~\ref{s:lamplighter}.

 Since exponential retainment holds trivialy for any finitely generated group, we define a new notion we dub  \textbf{strong} retainment that asks for protecting a positive portion of the graph (see Section~\ref{s:notation}). We end Section~\ref{s:lamplighter} showing
 that for the lamplighter on $\Z$, $LL(\Z):=\Z_2 \wr \Z$ with the switch-walk-switch generators, $f(n)=4\sqrt{2}^n$ is enough to get strong retainment.
  A slight modification of the proof of Theorem~\ref{t:lamplighter} shows that for a given Cayley graph of the lamplighter group there is no strong retainment for all functions $f$ with small enough exponential growth. Thus one gets exponential lower and upper bounds for strong retainment on $LL(\Z)$.

\bigskip

\noindent \textbf{Acknowledgments.}
  GA and MG were supported by the Israel Science Foundation Grant 957/20.
  RB has counted on the support of the Mathematical Institute of Leiden University.
  GK was supported by the Israel Science Foundation Grant 607/21 and by the Jesselson Foundation.

\section{Notation and preliminaries}\label{s:notation}
~
\par In this section, we introduce the notation that will be used throughout the text.

  Let $G$ be a finitely generated group, and fix $S$ a symmetric finite set of generators for $G$.
  We consider the Cayley graph of $(G,S)$ and identify its set of vertices with $G$.
  Denote by $e=e_{G}$ the identity element on $G$ and, for a given $R \in \Z \cap [0, \infty)$, let $B_{R}= B_{R}^{G} = B(e,R)$ be the ball centered in $e$ with radius $R$ with respect to the graph distance.
  Furthermore, consider the growth function $v_{G}: \mathbb{N}_{0} \to \mathbb{R}$ defined as
\begin{equation}
  v_{G}(R) = |B_{R}|.
\end{equation}
  When the group $G$ is clear from context, we will denote $v_{G}$ simply by $v$.

  Let $f: \mathbb{N} \to \mathbb{N}$ be a sequence of non-negative integers.
  Assume that, at time zero, a fire breaks on a finite subset $F_{0}$ of $G$.
  At subsequent integer times $n \in \mathbb{N}$, a set $W_{n}$ of size at most $f(n)$ vertices which are not on fire are declared safe.
  After that, spread the fire from $F_{n-1}$ to all neighboring sites that are not safe.
  Denote the set of vertices on fire at time $n$ by $F_{n}$ and by $U$ the set of vertices that are never on fire.
  We call $\{W_{n}\}_{n \in \mathbb{N}}$ an $\{f(n)\}$-allowed strategy if $|W_{n}| \leq f(n)$, for all $n \in \mathbb{N}$.
   We say that $G$ has the $\{f(n)\}$-fire retaining property if, for every choice of $F_{0}$, there exists an $\{f(n)\}$-allowed strategy such that the growth of $U$ with respect to the graph metric in $G$ is equivalent to the growth of $G$.

\nc{c:safe_growth}
\nc{c:safe_growth_2}

  The group $G$ has polynomial growth of order $d$ if
\begin{equation}
  v_{G}(n) \simeq n^{d},
\end{equation}
meaning that $\frac{v_{G}(n)}{n^{d}}$ is bounded from above and below by positive constants. The group $G$ has exponential growth if
\begin{equation}
  v_{G}(n) \simeq \gamma^{n},
\end{equation}
for some $\gamma>1$. Groups might also have intermediate growth, in the sense that it has super-polynomial and subexponential growth functions.

  For polynomial growth groups, $G$ has the $\{f(n)\}$-retaining property if, for any finite initial set $F_{0}$, there exist a constant $\uc{c:safe_growth}=\uc{c:safe_growth}(F_{0})>0$ and an $\{f(n)\}$-allowed strategy $\{W_{n}\}_{n \in \mathbb{N}}$ such that
\begin{equation}\label{eq:safe_growth}
  \left| U \cap B_{n} \right| \geq \uc{c:safe_growth}v_{G}(n),
\end{equation}
for all $n$ large enough.

  When $G$ is a group with intermediate or exponential growth, the fire retainment property is not equivalent to having a positive fraction of vertices not on fire as in~\eqref{eq:safe_growth}.
  Indeed, in the exponential-growth case for example, any two exponential functions have equivalent growth, which in particular implies that fire retainment can be achieved by protecting an amount of vertices that grows as a vanishing proportion of the graph and still has equivalent growth function.
  Motivated by this, we introduce the concept of strong retainment.
  In analogy to the polynomial-growth case, we say that a graph $G$ has the strong $\{f(n)\}$-retaining property if, for any finite initial set $F_{0}$, there exist a constant $\uc{c:safe_growth_2}=\uc{c:safe_growth_2}(F_{0})>0$ and an $\{f(n)\}$-allowed strategy $\{W_{n}\}_{n \in \mathbb{N}}$ such that
\begin{equation}\label{eq:safe_growth_2}
  \left| U \cap B_{n} \right| \geq \uc{c:safe_growth_2}v_{G}(n),
\end{equation}
for all $n$ large enough.
  Notice that strong retainment is equivalent to (weak) retainment if $G$ has polynomial growth (see~\eqref{eq:fire_growth}).

Mart\'{i}nez-Pedroza and Prytu{\l}a~\cite{mpp} proved that for polynomial growth groups, fire retaining is a quasi-isometry invariant in the sense that if, a Cayley graph $G$ satisfies $\{f(n)\}$-retainment and $H$ is quasi-isometric to $G$, then $H$ satisfies $\{h(n)\}$-retainment for some $h\simeq f$. The same remains true also for non-polynomial growth if one demands that $f$ is increasing. It is also quite straightforward that this is also the case if $f$ fixates on $0$ after finitely many steps (what is called in~\cite{mpp} ``constant step retainment'').

  For a set $A \subset G$, denote the vertex boundary of $A$ by
\begin{equation}\label{eq:vertex_boundary}
  \partial A = \{x \in G: \text{there exists } y \in A \text{ with } x \sim y\}.
\end{equation}

\section{Proof of Theorem~\ref{t:fire_retaining}}
~
\par This section contains the proof of Theorem~\ref{t:fire_retaining}. We retain the notation $F_n$ and $W_n$ from the previous section.
The key observation for the proof is that $\partial \left( \cup F_{n} \right) \subset \cup W_{n}$.
More precisely, we will use the relation
\begin{equation}\label{eq:boundary}
\partial F_{n} \subset \left(F_{n+1} \setminus F_{n} \right) \bigcup \cup_{k=0}^{n+1} W_{k}.
\end{equation}

We will also need an ``initialization step'' -- that if we do not have containment then for a large enough initial fire and time many vertices are on fire. Precisely,
\begin{lemma}Let $G$ be the Cayley graph of a group with polynomial growth of order $d$. Then for any $f(n)=o(n^{d-2})$ and any $R_0$ sufficiently large there exists a constant $\uc{c:containment}=\uc{c:containment}(f, R_{0})$ such that if the initial set $F_{0}$ is taken to be $B_{R_0}$ then for all $\{f(n)\}$-allowed strategies $\{W_{n}\}_{n \in \mathbb{N}}$,
\begin{equation}\label{eq:fire_growth}
  \left| F_{n} \right| \geq \uc{c:containment} n^{d}, \text{ for all } n \in \mathbb{N}.
\end{equation}
\end{lemma}
\begin{proof}
  This is proved exactly as in~\cite{abk} (see the proof of Theorem 1, and, in particular, Equation~(4) there). We skip the details.
\end{proof}

The core of the proof of Theorem~\ref{t:fire_retaining} is the following argument, which we will make precise below: assume that we have retainment but not containment, then for a large enough initial fire and time, both the set of vertices on fire $F_{n}$ and the set of vertices that are not on fire must hold a positive portion of the ball $B_n$. We then use an isoperimetric inequality to deduce that the boundary between these two sets is large \textit{inside the ball}. All the vertices in this boundary must either catch fire next turn or be protected (at that step or before). The upper bound on the number of vertices being protected now translates to a lower bound on the number of new vertices catching fire. This lower bound implies that in $O(n)$ steps we exhaust all vertices inside the ball, reaching a contradiction.\\


Note that the isoperimetry we need is a lower bound on the number of boundary vertices between $S\cap B_n$ and $B_n \setminus S$ for a given set $S$.  This is quite different than the usual isoperimetric inequality as sometimes most of the boundary vertices of $S$ may lie outside the ball. For instance, if the Cayley graph is a tree (as in the case of the free group) one may divide $B_n$ into two equal size sets with only one boundary vertex between them (compare with the fact that each of these sets has a large boundary of its own). The isoperimetric inequality we use is given in Lemma~\ref{lemma:isoperimetry}.
\nc{c:containment}
We are now ready to give the proof of the theorem.

\begin{proof}[Proof of Theorem~\ref{t:fire_retaining}]
  Assume the opposite, i.e., that there exists $f(n)=o(n^{d-2})$ such that, for all initial fires $F_{0}$, there exists an $\{f(n)\}$-allowed strategy for which $U$ has polynomial growth of order $d$.
  Choose $F_{0}$ such that~\eqref{eq:fire_growth} is true, and consider $\uc{c:growth}$ and $n_{0} \in \mathbb{N}$ such that~\eqref{eq:safe_growth} holds for all $n \geq n_{0}$.
  Furthermore, by choosing $n_{0}$ large enough, we can suppose that $F_{n} \subset B_{2n}$, for all $n \geq n_{0}$.

  Combining now Lemma~\ref{lemma:isoperimetry} with Equations~\eqref{eq:safe_growth} and~\eqref{eq:fire_growth} yields, for all $k \geq n$,
\begin{equation}
  \begin{split}
    \left| B_{3n} \cap \partial F_{k}\right| & \geq \frac{\uc{c:growth}}{n\cdot v(2n)} \left|F_{k} \cap B_{n}\right| \cdot \left|B_{n} \setminus F_{k}\right| \\
    & \geq \frac{\uc{c:growth}\uc{c:safe_growth}\uc{c:containment}}{n \cdot v(2n)} n^{2d} \\
    & \geq C n^{d-1},
\end{split}
\end{equation}
for some positive constant $C$.

  Fix $\tilde{C}$ such that, for all $n \in \mathbb{N}$,
\begin{equation}
  (\tilde{C}-1)\frac{C}{2}n^{d} \geq v(3n).
\end{equation}
  Notice that if we have, for all $k \in [n, \tilde{C} n]$,
\begin{equation}
  \left|\left(F_{k+1}\setminus F_{k}\right) \cap B_{3n} \right| > \frac{C}{2}n^{d-1},
\end{equation}
then
\begin{equation}
  \begin{split}
    \left| F_{\tilde{C} n} \cap B_{3n} \right| & \geq \sum_{k=n}^{\tilde{C} n} \left|\left(F_{k+1}\setminus F_{k}\right) \cap B_{3n} \right| \\
    & \geq \sum_{k=n}^{\tilde{C} n} \frac{C}{2}n^{d-1} \\
    & \geq \left(\tilde{C} n-n\right) \frac{C}{2}n^{d-1} \\
    & > v(3n),
  \end{split}
\end{equation}
a contradiction.

  This implies that there exists $k \in [n, \tilde{C} n]$ such that
\begin{equation}
  \left| \partial F_{k} \cap B_{3n} \right| \geq C n^{d-1} \text{ and } \left|\left(F_{k+1}\setminus F_{k}\right) \cap B_{3n} \right| \leq \frac{C}{2}n^{d-1}.
\end{equation}
  Together with Equation~\eqref{eq:boundary}, we obtain
\begin{equation}
  \left| \bigcup_{j=1}^{k} W_{j} \right| \geq \frac{C}{2}n^{d-1}.
\end{equation}
  However, we also have the bound
\begin{equation}
  \left| \bigcup_{j=1}^{k} W_{j} \right| \leq \sum_{j=1}^{k} f(j) \leq \sum_{j=1}^{\tilde{C}n}f(j) = o(n^{d-1}).
\end{equation}
  These two equations yield a contradiction and imply that $G$ does not have the $\{f(n)\}$-retaining property.
\end{proof}

\section{Proofs of Theorems~\ref{t:product_groups} and~\ref{t:product_groups_3}}
~
\par We now proceed with the proof of Theorems~\ref{t:product_groups} and~\ref{t:product_groups_3}.
  They follow basically the same strategy and for that reason we choose to present them together. We retain the notation $F_n$ and $W_n$ from Section \ref{s:notation}.

The base structure of both proofs as follows: first, we prove an initialization step (Lemma \ref{lemma:fire_spread} and Corollary \ref{cor:fire_spread}) showing that if we are well below the containment threshold then many vertices will be on fire. Then, we examine how the fire spreads on well-chosen fibers of the Cartesian product $G \times H$, alternating between long stretches where we follow only the spread on fibers of the form $G\times {h}$ and ${g}\times H$. Assumptions~\eqref{eq:slow_growth_f} and~\eqref{eq:slow_growth_f_2} allow us to apply the pigeonhole principle in order to prove the existence of such fibers. The final important observation is that, if there exists a vertex $x \in G$ such that the fiber $\{x\} \times H$ eventually contains both a vertex on fire and a vertex not on fire, then it necessarily contains at least one protected vertex.

\begin{lemma}\label{lemma:fire_spread}
  Assume $G$ is a group that does not satisfy the $\{g(n)\}$-containment property and let $f \leq g$.
  There exists a finite set $F_{0}$ of vertices initially on fire such that, for any $\{f(n)\}$-allowed strategy, we have
\begin{equation}
  |F_{n}| \geq \sum_{k=1}^{n}g(k)-f(k), \text{ for all } n \in \mathbb{N}.
\end{equation}
\end{lemma}

\begin{proof}
  Since $G$ does not satisfy the $\{g(n)\}$-containment property, there exists a finite set $F_{0}$ for which we necessarily have
\begin{equation}
  |F_{n+1} \setminus F_{n}| \geq g(n)-f(n), \text{ for all } n \in \mathbb{N}.
\end{equation}
  Indeed, if this were not the case and, for all $F_{0}$, there exists some $n$ such that the equation above does not hold, we can protect $g(n)$ vertices at step $n$ and contain the fire.

  The proof is now complete by noticing that
\begin{equation}
  |F_{n}| \geq \sum_{k=1}^{n} |F_{k} \setminus F_{k-1}| \geq \sum_{k=1}^{n} g(k)-f(k).
\end{equation}
\end{proof}

\nc{c:fire_spread}

\begin{cor}\label{cor:fire_spread}
  Assume that $G$ has exponential growth. Then, there exists a constant $\uc{c:fire_spread}>0$ such that, for every function $f(n) \leq e^{\uc{c:fire_spread} n}$, there exists a finite set $F_{0}$ of vertices initially on fire such that, for any $\{f(n)\}$-allowed strategy, we have
\begin{equation}
  |F_{n}| \geq e^{\uc{c:fire_spread} n}, \text{ for all } n \in \mathbb{N}.
\end{equation}
\end{cor}

\begin{proof}
  This is a consequence of Lemma~\ref{lemma:fire_spread} together with the fact that the graph $G$ does not satisfy exponential containment for any rate smaller than the growth rate of $G$ (see Lehner~\cite{ln}).
\end{proof}

  We are now ready to prove Theorem~\ref{t:product_groups}.
\begin{proof}[Proof of Theorem~\ref{t:product_groups}]
  Choose $F_{0}$ as in Corollary~\ref{cor:fire_spread} and declare all vertices in $F_{0} \times \{e_{H}\}$ as initially on fire in $G \times H$.
  Set $k = \max \{|x|: x \in F_{0} \}$, and choose $n \geq k$.
  Consider an $\{f(n)\}$-allowed strategy for the fire-retaining problem on $G \times H$, and denote by $F_{n}$ the number of vertices on fire at step $n$.
  Observe that, due to Corollary~\ref{cor:fire_spread}, we have
\begin{equation}
  \Big|F_{n} \cap \big( G \times \{e_{H}\} \big) \Big| \geq e^{\uc{c:fire_spread} n}, \text{ for all } n \in \mathbb{N}.
\end{equation}

  Denote now by $P_{n} = \bigcup_{k=1}^{10n}W_{n}$ the collection of protected vertices up to time $10n$, and notice that
\begin{equation}
  |P_{n}| \leq \sum_{i=1}^{10n} f(i) = o(v_{H}(n)).
\end{equation}
  In particular, there exists at least one $x\in G$ such that $(x,e) \in F_{n} $ and such that the set $\{x\} \times B^{H}_{n}$ has no protected vertices up to time $3n$.
  This then implies that, by running the process up to time $3n$, all vertices in $\{x\} \times B^{H}_{n}$ are on fire.

  Once again using the fact that $|P_{n}| = o(v_{H}(n))$, there exists $y \in B^{H}_{n}$ such that
\begin{equation}
  P_{n} \cap \big( G \times \{y\} \big) = \emptyset.
\end{equation}
  This then implies that, by time $5n$, all vertices on $B^{G}_{2n}(x) \times \{y\}$ are on fire, which yields that all vertices on $B^{G}_{n} \times \{y\}$ are also on fire.

  We now run the process up to time $9n$.
  Since the whole set $B_{n}^{G} \times \{y\}$ is on fire at time $5n$, if $(a,b) \in B^{G}_{n} \times B^{H}_{n}$ is not on fire at time $9n$, then
\begin{equation}
  P_{n} \cap \big( \{a\} \times B^{H}_{n} \big) \neq \emptyset.
\end{equation}
  This then implies that, if $U$ denotes the set of vertices that are never on fire,
\begin{equation}
\Big| U \cap \big( B^{G}_{n} \times B^{H}_{n} \big) \Big| \leq v_{H}(n) |P_{n}| = o(v_{H}(n)^{2}).
\end{equation}
Since this is subexponential, it is asymptotically slower than the growth of $G\times H$ and the proof is and concluded.
\end{proof}

The proof of Theorem~\ref{t:product_groups_3} follows essentially the same steps.
\begin{proof}[Proof of Theorem~\ref{t:product_groups_3}]
  Due to~\eqref{eq:slow_growth_f_2}, we can choose $F_{0} \subset G$ and $\tilde{F}_{0} \subset H$ as in Corollary~\ref{cor:fire_spread}. Declare all vertices in $\big( F_{0} \times \{e_{H}\} \big) \cup \big(\{e_{G}\} \times \tilde{F}_{0} \big) $ as initially on fire in $G \times H$.

  Denote by $P_{n}$ the collection of protected vertices up to time $10n$, and write $P_{G}$ and $P_{H}$ to its projection on $G$ and $H$, respectively.

  Proceeding as in the proof of Theorem~\ref{t:product_groups} (using the fact that $f$ has subexponential growth), if $(x,y) \in B^{G}_{n} \times B^{H}_{n}$ is not on fire by time $10n$, then $x \in P_{G}$ and $y \in P_{H}$.
  In particular, if $U$ denotes the set of vertices never on fire, then
\begin{equation}
  \Big| U \cap \big( B^{G}_{n} \times B^{H}_{n} \big) \Big| \leq |P_{G} \times P_{H}| \leq \Big( \sum_{i=1}^{10n} f(i) \Big)^{2} = o(e^{Kn}),
\end{equation}
for every $K >0$, which concludes the proof.
\end{proof}

\section{Fire retaining on lamplighter groups}\label{s:lamplighter}

Given two discrete groups $G,H$, the (restricted) wreath product $H \wr G$ is the semidirect product $G \ltimes \bigoplus_G H$. Each element in the wreath product is a pair $(g,L)$ with $g\in G$ and $L\in \bigoplus_G H$. Recall  that elements in $\bigoplus_G H$ are functions from $G$ to $H$ that differ from the identity in only finitely many places. The product rule is given by $(g_1,L_1)(g_2,L_2)=(g_1g_2,L_1 L_2^{g_1})$, where $L_1 L_2^{g_1} ( s) = L_1(s)L_2(g_1^{-1}s)$.  Wreath products play an important role in geometric group theory and often serve as a rich source of examples.

The name \textit{lamplighter groups} comes from thinking of the coordinate $g\in G$ as the position of a lamplighter that walks on $G$ and may change the lamps by changing the lamp at his position. We use this terminology and call $L$ the  lamp configuration. Fix two generating sets  $S_G, S_H$  of $G,H$ respectively. We refer to elements of the form $(s_g,e_H)$ with $s_g\in S_G$ as ``walk'' and elements of the form $(e_G,\delta_h)$, for $h \in S_H$, with $\delta_h(e_G)=h$ and $\delta_h(g)=e_H$ for $g \neq e_G$  as switches. Unless otherwise stated we will take as a generating set the switch-walk-switch elements, that is elements of the form $S_1WS_2$, where $S$ switches lamps and $W$ is a walk generator. By quasi-isometry invariance this choice does not matter for the sake of the main theorem.

\begin{proof}[Proof of Thereom \ref{t:lamplighter}]
  Denote $\Gamma=H \wr G$, and let $f$ be some subexponential function. By Corollary \ref{cor:fire_spread}
  given a large enough initial fire $F_0$, for all large enough $n$ there are exponentially many vertices on fire inside $B_n$.
Assume by contradiction that $\Gamma$ does satisfy $\{f(n)\}$-retainment. Then for $F_0$ there is some $\{f(n)\}$- allowed strategy for which for  all large enough $n$ the set of vertices never on fire inside $B_n$ is also exponentially large.
From now on we fix this $\{f(n)\}$-allowed strategy.
All together we conclude that there exists some $c>0$ such that, for all large enough $n$, $|U\cap B_n(\Gamma)|,|F_n \cap B_n(\Gamma)| \geq e^{cn}$.
To reach a contradiction, we will find exponentially large subsets $A=\{a_i\}\subset (F_n \cap B_n)$ and $B=\{b_i\}\subset (U\cap B_n)$ with some good properties, and a set of paths $\gamma_i:a_i \rightarrow b_i$ so that these paths are pairwise disjoint (i.e. $\gamma_i$ and $\gamma_j$ do not intersect for $i\neq j$) and the length of each path $\gamma_i$ is at most $100n$.

Before constructing the sets $A,B$ and the paths $\gamma_i$, let us show how we use these paths to reach a contradiction from which the theorem follows. Let $P$ denote the set of all protected vertices till time $101n$. Since $f$ is subexponential, for any large enough $n$, $|P|<|A|$. Therefore there is at least one $i$ for which no vertex in $\gamma_i$ is protected until time $101n$. Since $a_i$ is on fire at time $n$, the fire will spread unhindered through $\gamma_i$ until it reaches $b_i$ before time $n+100n$. Thus $b_i$ will eventually catch fire contradicting the fact that $b_i\in U$.

It remains now to construct the sets $A,B$ and the paths $\gamma_i$.
Since $U\cap B_n$ and $F_n\cap B_n$ are exponential in size, by a simple pigeonhole principle we can always find subsets $A=\{a_i\}_{i\in I}\subset F_n\cap B_n$ and $B =\{b_i\}_{i\in I} \subset U \cap B_n$ of exponential size (i.e. $|A|,|B|\geq e^{c_1n}$ for some suitable $c_0<c_1<c$ and any large enough $n$) such that one of the following four cases holds:
\begin{enumerate}
\item \label{case:different lamps} All elements in $A\cup B$ have different lamp configurations.
\item \label{case:different lamplighters} All elements in $A\cup B$ have different lamplighter positions.
\item \label{case:mixed} All elements in $A$ have different lamp configurations, all elements in $B$ have the same lamp configuration and  different lamplighter positions, and the lamp configuration of the elements of $B$ does not appear in the elements of $A$.
\item \label{case:symmetric} All elements in $B$ have different lamp configurations, all elements in $A$ have the same lamp configuration and  different lamplighter positions, and the lamp configuration of the elements of $A$ does not appear in the elements of $B$.
\end{enumerate}
(Note that if $G$ is of subexponential growth, we could always find subsets satisfying the first clause, but in the general case this is not true \emph{a priori}).
Before going over the different cases, let us set some notation. The elements of $A$  will be denoted $a_i=(x_i,k_i)$ where $x_i$ is the lamp configuration (that is supported on $B_n$) and $k_i$ is the lamplighter position in $G$. Similarly, elements of $B$ are $b_i=(y_i,l_i)$ with lamps $y_i$ and lamplighter positions $l_i$.
For an element $g\in G$ we denote by $\overrightarrow{g}$ some (arbitrarily chosen) fixed geodesic path from $e_G$ to $g$ in $G$, and for an element $k\in G$ we denote by $k\overrightarrow{g}$ the translation of the path $\overrightarrow{g}$ by $k$ (i.e. making the same steps only starting from $k$ instead of $e_G$). Finally, we fix two elements $g,g^*\in G$ with $\dist(g,e_G)=5n$ and $\dist(g^*,e_G)=10n$, and choose some non-identity generator $h$ of $H$.
We are now ready to go over the cases one by one, constructing the paths $\gamma_i$ and diluting the sets if needed.
\begin{enumerate}
\item
Case \ref{case:different lamps}: If the ``all off'' configuration is in one of $A$ or $B$, remove it. We construct the paths $\gamma_i:a_i\rightarrow b_i$ as follows
\begin{enumerate}
\item Move the lamplighter to $g$ by following some geodesic.
\item Make a copy of the lamp configuration $x_i$ around $g$, so that the lamps in $B_n(g)$ are just a translation by $g$ of the lamps in $B_n$.
(this can be achieved simply by multiplying by $a_i$).
\item Move the lamplighter to $k_i g^*$ along some geodesic.
\item Make another copy of the lamp configuration $x_i$ around $g^*$.
\item Move the lamplighter to $e_G$ along some geodesic.
\item Change the lamps in $B_n$ from $x_i$ to $y_i$ (in the shortest way).
\item Move the lamplighter to $g$.
\item Turn off all the lamps in $B_n(g)$.
\item Move the lamplighter to $g^*$ (along some geodesic in $G$).
\item Turn off all the lamps in $B_n(g^*)$.
\item Move the lamplighter to $l_i$ reaching $b_i$.
\end{enumerate}
We claim that the paths $\gamma_i$ are disjoint for $i\neq j$. To see this, note that at any step of the path, either the configuration inside $B_n(e_G)$ is $x_i$ (steps $1$-$5$) or the lamp configuration in $B_n(g)$ is a copy of $x_i$ (steps $3-7$) or the lamp configuration in $B_n(g^*)$ is a copy on $x_i$ (steps $5-9$) or the lamp configuration in $B_n(e_G)$ is $y_i$ (steps $7-11$). Since the $x_i's$ and $y_i's$ are all distinct for different $i$'s, a vertex in $\gamma_i$ cannot be in $\gamma_j$.
To bound the length of $\gamma_i$ note that each ``move'' step takes at most $11n$ steps, and each ``copy/change the lamps'' step takes at most $2n$ steps.
\\
\item Case \ref{case:different lamplighters}: We construct the paths $\gamma_i:a_i\rightarrow b_i$ as follows
\begin{enumerate}
\item Move the lamplighter to $k_i g$ along the path $k_i\overrightarrow{g}$.
\item Light the lamp at $k_i g$ to value $h$.
\item Move the lamplighter to $l_ig^*$ along some geodesic.
\item Light the lamp at $l_ig^*$ to value $h$.
\item Move the lamplighter to $k_i g$ along some geodesic.
\item Turn off the lamp at $k_i g$.
\item Use some geodesic in $\Gamma$ to turn the lights in $B_n(e_G)$ to $y_i$.
\item Move the lamplighter to $l_i g^*$ along some geodesic.
\item Turn off the lamp at $l_i g^*$.
\item Move the lamplighter backwards along the path $l_i \overrightarrow{g^*}$ to $l_i$, reaching $b_i$.
\end{enumerate}
Once again the length of the paths created is bounded by $100n$ as the geodesics in $G$ we move along are of length at most $11n$, and changing the lamps in $B_n(e_G)$ to $y_i$ takes at most $2n$ steps once the lamplighter is in $e_G$.\\
In this construction we do not claim that all the paths we constructed are disjoint, but that each path $\gamma_i$ only intersects $\leq 1000n^2$ other paths. Thus filtering out any path $\gamma_i$ that intersects some $\gamma_j$ for $j<i$ will leave us with an exponential family of disjoint paths (and diluted yet exponential families $A'\subset A$ and $B'\subset B$ with the corresponding indexes), completing the case.
To see the claim, first note that since $k_i,k_j, k_i g,k_j g,  l_i,l_j,  l_i g^*, l_j g^*$ are all different for $i\neq j$, the paths cannot intersect while the lamp at $k_i g$ or $l_i g^*$ is on.
Thus we need only worry about vertices of step $1$ or step $10$, overlapping with vertices from the same steps for different $j$'s. For this, enumerate the vertices on the paths $\overrightarrow{g},\overrightarrow{g^*}$. Let $g_m,g^*_m$ denote the $m$th vertex in $\overrightarrow{g},\overrightarrow{g^*}$. Note that since these are geodesics, $g_r\neq g_q$ and $g^*_r\neq g^*_q$ for any $r\neq q$. Using inverses we conlcude that for any $i,j$ there is at most one pair of indeces $r,q$ so that $k_i g_r = k_j g_q$, and therefore that the vertices of step $1$ of $\gamma_i$ can intersect at most $25n^2$ different paths $\gamma_j$. The same  argument works for step $10$ giving a bound of $100n^2$. This concludes the construction for this case.

\item Cases \ref{case:mixed} and \ref{case:symmetric}. Without loss of generality, we may assume all the $a_i$'s have different lamp configurations, different from $y$ which is the common lamp configuration for all $b_i\in B$, which in turn each have different lamplighter positions. If roles are reversed we simply reverse the roles in the construction.
    We construct the paths $\gamma_i$ from $b_i$ to $a_i$ as follows:
    \begin{enumerate}
      \item Move the lamplighter along $\overrightarrow{g}$ to $l_i g$.
      \item Light the lamp at $l_i g$ to value $h$.
      \item Use some geodesic to change the lamps in $B_n$ to $x_i$.
      \item Use some geodesic to get the lamplighter to $l_i g$.
      \item Turn off the lamp at $l_i g$.
      \item Use some geodesic to get the lamplighter to $k_i$ reaching $a_i$.
    \end{enumerate}
    As in the previous case, since $l_i\neq l_j$ for $i\neq j$, the paths $\gamma_i,\gamma_j$ can only intersect when the lamps at $l_i g$ and $l_j g $ are off. Since the configurations $x_i$ and $x_j$ differ, it remains that intersections between $\gamma_i$ and $\gamma_j$ can only happen from vertices at step $1$ of the path. The same argument as before now gives that each $\gamma_i$ intersects at most $25n^2$ different $\gamma_j$'s.
\end{enumerate}
    The proof now concludes by diluting the set of paths as before.
\end{proof}

We finish this section by showing that it is possible to get strong retainment in the standard lamplighter group on $\Z$ by protecting an exponential number of vertices with exponent strictly small than the growth rate.
\begin{theo}
  Let $\Gamma$ be the Cayley graph of the lamplighter group $\Z_2 \wr \Z$ w.r.t. the  switch-walk-switch generating set. Then $\Gamma$ satisfies $f(n)=2^{(n+2)/2}$ strong retainment.
\end{theo}

\begin{proof}
Given an initial fire $F_0\subset \Gamma$, we can find some $r$ such that  $F_0\subset B_r$ and choose a parameter $M>r$.
The set we are going to save (that will never burn) will be the set
\begin{equation}
S=\Big\{ \big( (a_i),k \big) \in \Gamma \, :
\begin{array}{c}
 a_0=a_1=a_2=\dotsb=a_M =1, \\
 a_i = 0, \text{ for all } i<0, \text{ and } k>M+1
\end{array}
\Big\}.
\end{equation}
A straightforward calculation shows that $\frac{|S\cap B_R|}{|B_R|} \rightarrow 2^{-M-3}$ as $R\rightarrow \infty$.\\
To save this set, the vertices we protect will be
\begin{equation}
P=\partial_v^{in} S = \Big\{ \big( (a_i),k \big) \in \Gamma\, :
\begin{array}{c}
 a_i=0, \text{ for all } i<0, \\
 a_0=\dotsb=a_M=1, \text{ and } k=M+2
\end{array}
\Big\}.
\end{equation}
We will protect these vertices in lexicographical order over the rest of the lamps $a_{M+1}$, $a_{M+2}$ ..., protecting by time $n$ all vertices of the form $P_n= \big\{ \big( (a_i),k \big) \in P \, : \, a_i=0 \, \text{ for all } i > M+n/2 \big\}$. This can be done by protecting $2^{(n+2)/2}$ vertices on turn $n$.
It now remains to show that the set $S$ never catches fire. Assume by contradiction that some vertices in $S$ do catch fire, and let $v$ be one of the vertices of $S$ that catches fire \textit{first}, and denote by $n$ the time at which it catches fire. Then $v\in \partial_v^{in} S=P$, and in particular $v$ is of the form $v=(\underbar{a},M+2)$. Since $v$ caught fire at time $n$, it cannot belong to $P_n$, hence there is some $i>M+n/2$ for which $a_i=1$. This implies $\dist(v,e)>M+n$ which in turn implies $\dist(v,F_0)>n$, contradicting the fact that $v$ caught fire at time $n$.
\end{proof}

\appendix

\section{Isoperimetry}\label{app:isoperimetry}
~
\par For a function $f:G \to \R$, define its discrete derivative at a given oriented edge $e=(e^{-}, e^{+})$ as
\begin{equation}
\nabla f (e) = f(e^{+})-f(e^{-}).
\end{equation}
Moreover, for a subset $B \subset G$, let
\begin{equation}
|| f ||_{L^{1}(B)} = \sum_{x \in B} |f(x)|,
\end{equation}
and
\begin{equation}
|| \nabla f ||_{L^{1}(B)} = \sum_{e \in \vec{E}(B)} |\nabla f(e)|,
\end{equation}
where $\vec{E}(B)= \{e=(e^{-},e^{+}): \, e^{-}, e^{+} \in B \text{ and } \{e^{-}, e^{+}\} \in E(G) \}$.
\begin{lemma}\label{lemma:poincare}
Let $G$ be a finitely generated group. For every integer $R \geq 1$ and function $f:B_{3R} \to \R$,
\begin{equation}
|| f-f_{R} ||_{L^{1}(B_{R})} \leq 2R\frac{v(2R)}{v(R)} ||\nabla f||_{L^{1}(B_{3R})},
\end{equation}
where $f_{R}$ denotes the average of $f$ over $B_{R}$
\begin{equation}
f_{R}= \frac{1}{v(R)}\sum_{x \in B_{R}}f(x).
\end{equation}
\end{lemma}

\par The proof follows the steps from~\cite{kleiner}, which proves an analogous inequality for the $L^{2}$ norm.
\begin{proof}
Fix an integer $R \geq 1$. Define, $\delta f: B_{3R} \to \R$ as
\begin{equation}
\delta f(x)=\sum_{y \sim x, \, y \in B_{3R}}|f(y)-f(x)|,
\end{equation}
and notice that $||\delta f||_{L^{1}(B_{3R})} = ||\nabla f||_{L^{1}(B_{3R})}$.
For every $y \in G$, consider a path of minimal length $\gamma_{y}:\{1, \dots, |y|\} \to G$ from the identity element $e \in G$ to $y$. If $y \in B_{2R}$, then
\begin{equation}
\sum_{x \in B_{R}} \sum_{i=1}^{|y|}\delta f(x\gamma_{y}(i)) \leq 2R \sum_{z \in B_{3R}} \delta f(z),
\end{equation}
since the map from $B_{R} \times \{1, \dots |y|\}$ to $B_{3R}$ that maps $(x,i)$ to $x\gamma_{y}(i)$ is at most $2R$ to one.

Observe now that, for each $x \in B_{R}$,
\begin{equation}
\begin{split}
\sum_{x \in B_{R}}|f(x)-f_{R}| & \leq \frac{1}{v(R)}\sum_{x \in B_{R}}\sum_{y \in B_{R}}|f(x)-f(y)| \\
& \leq \frac{1}{v(R)}\sum_{x \in B_{R}} \sum_{y \in B_{R}} \sum_{i=1}^{|x^{-1}y|} \delta f(x\gamma_{x^{-1}y}(i)) \\
& \leq \frac{1}{v(R)}\sum_{x \in B_{R}} \sum_{w \in B_{2R}} \sum_{i=1}^{|w|} \delta f(x\gamma_{w}(i)) \\
& \leq \frac{1}{v(R)}\sum_{w \in B_{2R}} 2R \sum_{z \in B_{3R}} \delta f(z) \\
& \leq 2R \frac{v(2R)}{v(R)}  ||\nabla f||_{L^{1}(B_{3R})},
\end{split}
\end{equation}
concluding the proof.
\end{proof}

\nc{c:growth}

  The lemma above can be used to deduce isoperimetric inequalities for subsets of balls of Cayley graphs.
  Recall from~\eqref{eq:vertex_boundary} that $\partial A$ denotes the vertex boundary of $A \subset G$.
  If we consider $f=\charf{A}$, it is a straightforward calculation that
\begin{equation}\label{eq:nabla}
||\nabla f||_{L^{1}(B_{R})} \leq 2|S| \cdot \left|B_{R} \cap \partial A\right|.
\end{equation}
Besides,
\begin{equation}\label{eq:average}
|| f-f_{R} ||_{L^{1}(B_{R})} = \frac{2}{v(R)}|B_{R} \setminus A| \cdot |A \cap B_{R}|.
\end{equation}
Combining the two equations above can be combined with Lemma~\ref{lemma:poincare} we write
\begin{align*}
2|S| \cdot \left|B_{3R} \cap \partial A\right|
&\stackrel{\textrm{(\ref{eq:nabla})}}{\ge}
||\nabla f||_{L^{1}(B_{3R})}
\stackrel{\textrm{\ref{lemma:poincare}}}{\ge}
\frac{v(R)}{2Rv(2R)}|| f-f_{R} ||_{L^{1}(B_{R})}\\
&\stackrel{\textrm{(\ref{eq:average})}}{=}
\frac{1}{Rv(2R)}|B_{R} \setminus A| \cdot |A \cap B_{R}|.
\end{align*}
We have thus obtained the following isoperimetric inequality.
\begin{lemma}\label{lemma:isoperimetry}
There exists a positive constant $\uc{c:growth}>0$ such that, for all subsets $A \subset G$,
\begin{equation}
\left|B_{3R} \cap \partial A\right| \geq \frac{\uc{c:growth}}{R \cdot v(2R)}|B_{R} \setminus A| \cdot |A \cap B_{R}|.
\end{equation}
\end{lemma}

\bibliographystyle{plain}
\bibliography{mybib}

\end{document}